\documentclass[a4paper,reqno]{amsart}
\usepackage[margin=2.5cm,papersize={19.75cm,28cm}]{geometry}
\usepackage{amsfonts}
\usepackage{amsmath}
\usepackage{cite}
\usepackage{amsopn}
\usepackage{amssymb}
\usepackage{colonequals}
\usepackage{paralist}
\usepackage{epstopdf}
\usepackage{graphicx}
\usepackage{tikz}
\ifpdf
  \DeclareGraphicsExtensions{.eps,.pdf,.png,.jpg}
\else
  \DeclareGraphicsExtensions{.eps}
\fi

\numberwithin{equation}{section}

\newtheorem{theorem}{Theorem}
\newtheorem{proposition}{Proposition}
\newtheorem{corollary}{Corollary}
\newtheorem{lemma}{Lemma}
\newtheorem{definition}{Definition}

\theoremstyle{definition}
\newtheorem*{remarkth}{Remark}
\newenvironment{remark}{\begin{remarkth}}{\hfill$\lozenge$\end{remarkth}}

\DeclareMathOperator{\ad}{ad}
\DeclareMathOperator{\Ad}{Ad}

\DeclareMathOperator{\vecspan}{span}
\DeclareMathOperator{\tr}{tr}

\allowdisplaybreaks

\begin{document}

\title[Optimal Control Problems with Symmetry Breaking Cost Functions]{Optimal Control Problems with Symmetry Breaking Cost Functions}

\author[A. M. Bloch]{Anthony M. Bloch}
   \address{A. M. Bloch: Department of Mathematics, University of Michigan, Ann Arbor, MI 48109-1043, USA}  \email{abloch@umich.edu}
\author[L. J. Colombo]{Leonardo J. Colombo} 
\address{L. J. Colombo: Department of Mathematics, University of Michigan, Ann Arbor, MI 48109-1043, USA}\email{ljcolomb@umich.edu} 
\author[R. Gupta]{Rohit Gupta}
   \address{R. Gupta: Department of Aerospace Engineering, University of Michigan, Ann Arbor, MI 48109-2140, USA}  \email{rohitgpt@umich.edu}
\author[T. Ohsawa]{Tomoki Ohsawa}
   \address{T. Ohsawa: Department of Mathematical Sciences, The University of Texas at Dallas, Richardson, TX 75080-3021, USA}  \email{tomoki@utdallas.edu}

\begin{abstract}
We investigate symmetry reduction of optimal control problems for left-invariant control systems on Lie groups, with partial symmetry breaking cost functions. Our approach emphasizes the role of variational principles and considers a discrete-time setting as well as the standard continuous-time formulation. Specifically, we recast the optimal control problem as a constrained variational problem with a partial symmetry breaking Lagrangian and obtain the Euler--Poincar{\'e} equations from a variational principle. By applying a Legendre transformation to it, we recover the Lie--Poisson equations obtained by A. D. Borum [\textit{Master's Thesis}, University of Illinois at Urbana-Champaign, 2015] in the same context. We also discretize the variational principle in time and obtain the discrete-time Lie--Poisson equations. We illustrate the theory with some practical examples including a motion planning problem in the presence of an obstacle.

\end{abstract}
\maketitle


\vbox{\raggedleft AMS s.\,c.\,(2010): 70G45, 70H03, 70H05, 37J15, 49J15}


\section{Introduction}
Symmetry reduction of optimal control problems (OCPs) for left-invar\-iant control systems on Lie groups has been studied extensively over the past couple of  decades. Such symmetries are usually described as an invariance under an action of a Lie group and the system can be reduced to a lower-dimensional one or decoupled into subsystems by exploiting the symmetry (see, e.g., \cite{Agrachev2004}, \cite{Bloch1994}, \cite{Leon2004}, \cite{Grizzle1984}, \cite{Koon1997}, \cite{Krishnaprasad1993}, \cite{Ohsawa2013}, , \cite{Sachkov2009}, \cite{Van1987}). Symmetry reduction of OCPs is desirable from a computational point of view as well. Given that solving OCPs usually involves iterative/numerical methods such as the shooting method (as opposed to solving a single initial value problem), reducing the system to a lower-dimensional one results in a considerable reduction of the computational cost as well.

The goal of this work is to study symmetry reduction of OCPs for left-invariant control systems on Lie groups with partially broken symmetries, more specifically, cost functions that break some of the symmetries but not all. Symmetry breaking is common in several physical contexts, from classical mechanics to particle physics. The simplest example is the heavy top dynamics (the motion of a rigid body with a fixed point in a gravitational field), where due to the presence of gravity, we get a Lagrangian that is $\mathrm{SO}(2)$-invariant but not $\mathrm{SO}(3)$-invariant, contrary to what happens for the free rigid body. Reduction theory for Lagrangian/Hamiltonian systems on Lie groups with symmetry breaking was developed in \cite{Holm1998}, \cite{Marsden1984_2}, \cite{Marsden1984_1}. In the context of motion planning, the symmetry breaking appears naturally in the form of a barrier function as we shall see in Section \ref{subsec4.2}.

The results in this paper are the Lagrangian/variational counterpart of those in \cite{Borum2015}, \cite{Borum2014}; we also develop a discrete-time version of the results. From the Lagrangian point of view, we obtain the Euler--Poincar{\'e} equations from a variational principle and by using a Legendre transformation, we obtain the Lie--Poisson equations. By discretizing the variational principle in time, we obtain the discrete-time Lie--Poisson equations.

We also work out some practical examples to illustrate the theory, including an optimal control description of the heavy rigid top equations as well as a motion planning problem with obstacles; the latter provides an example of the symmetry breaking by a barrier function mentioned above.

\subsection{Outline}
The paper is organized as follows. In Section \ref{sec2}, we introduce some preliminaries about geometric mechanics on Lie groups. In Section \ref{sec3}, we study the Euler--Poincar{\'e} reduction of OCPs for left-invariant control systems on Lie groups, with partial symmetry breaking cost functions and we also obtain the Lie--Poisson equations by using a Legendre transformation. Furthermore, we consider the discrete-time case and obtain the discrete-time Lie--Poisson equations. In Section \ref{sec4} some practical examples are considered to illustrate the theory. Finally, in Section \ref{sec5}, we make some concluding remarks.

\section{Preliminaries}\label{sec2}
Let $\mathcal{Q}$ be an $n$-dimensional differentiable manifold with local coordinates $(q^{1},\hdots,q^{n})$, the configuration space of a mechanical system. Denote by $T\mathcal{Q}$ its tangent bundle with induced local coordinates $(q^{1},\hdots,q^{n},\dot{q}^{1},\hdots,\dot{q}^{n})$. Given a Lagrangian $L\colon T\mathcal{Q}\to\mathbb{R}$, the corresponding Euler--Lagrange equations are 
\begin{align}
\frac{d}{dt}\left(\frac{\partial L}{\partial\dot q^i}\right)-\frac{\partial L}{\partial q^i} = 0, \hspace{5pt} i = 1,\hdots,n,\nonumber
\end{align}
which determine a system of implicit second-order differential equations. 

An indispensable tool in the study of mechanical systems is symmetry reduction. Particularly, when the configuration space is a Lie group, one can reduce the Euler--Lagrange equations, which generally give a system of second-order equations, to a system of first-order ones on its Lie algebra. 

Let $G$ be a Lie group and $\mathfrak{g}$ be its Lie algebra. Let $L_{g}: G\to G$ be the left translation by an element $g\in G$, i.e., $L_{g}(h) = gh$, for any $h\in G$. Note that the map $L_g$ is a diffeomorphism on $G$ and is also a left action of $G$ on itself. The tangent map  of $L_g$ at $h\in G$ is denoted by $T_{h}L_{g}\colon T_{h}G\to T_{gh}G$. Similarly, the cotangent map of $L_g$ at $h\in G$ is denoted by $T_{h}^{*}L_{g}\colon T^{*}_{gh}G\to T^{*}_{h}G$. A vector field $X\in\mathfrak{X}(G)$ is called left-invariant, if $X(L_{g}(h)) = T_{h}L_{g}(X(h))$, for any $g$, $h\in G$. Let $\Phi\colon G\times\mathcal{Q}\to \mathcal{Q}$ be a left action of $G$ on $\mathcal{Q}$. A function $f\colon \mathcal{Q}\to\mathbb{R}$ is said to be $G$-invariant, if $f\circ\Phi_{g} = f$, for any $g\in G$.

If we assume that the Lagrangian $L\colon TG\to\mathbb{R}$ is $G$-invariant under the left action of $G$ on $TG$, then we can obtain a reduced Lagrangian $\ell\colon\mathfrak{g}\to\mathbb{R}$, where $\ell(\xi) = L(g^{-1}g,T_{g}L_{g^{-1}}(\dot{g}))$ $= L(e,\xi)$. We can now obtain the reduced Euler--Lagrange equations, commonly known as the Euler--Poincar{\'e} equations (see, e.g., \cite{Bloch2015}, \cite{Holm2009}, \cite{Marsden1999}), which are given by\footnote{For an infinite-dimensional Lie algebra, one generally uses the $\delta$ notation instead of the $\partial$ notation (see, e.g., \cite{Bloch1996}).}
\begin{align}
\frac{d}{dt}\frac{\partial\ell}{\partial\xi} = \ad^{*}_{\xi}\frac{\partial\ell}{\partial\xi}.\label{eq_ep_intro}
\end{align}
The Euler--Poincar{\'e} equations together with the reconstruction equation $\xi = T_{g}L_{g^{-1}}(\dot{g})$ are equivalent to the Euler--Lagrange equations on $G$. If we assume that the reduced Lagrangian $\ell$ is hyper-regular, then we can obtain the reduced Hamiltonian $h\colon \mathfrak{g}^{*}\to\mathbb{R}$ (by using a Legendre transformation) given by
\begin{align}
h(\mu) = \langle\mu,\xi\rangle-\ell(\xi),\nonumber
\end{align}
where $\mu = \frac{\partial\ell}{\partial\xi}\in\mathfrak{g}^{*}$. The Euler--Poincar{\'e} equations \eqref{eq_ep_intro} can now be written as the Lie--Poisson equations (see, e.g., \cite{Bloch2015}, \cite{Holm2009}, \cite{Marsden1999}), which are given by
\begin{align}
\dot{\mu} = \ad^{*}_{\frac{\partial h}{\partial\mu}}\mu.
\end{align}

\section{Optimal Control Problems on Lie Groups}\label{sec3}
We will first define a left-invariant control system on $G$, which is assumed to be $n$-dimensional.

\begin{definition} 
A left-invariant control system on $G$ is given by
\begin{align}
\dot{g} = T_{e}L_{g}(u),\nonumber
\end{align}
where $g(\cdot)\in C^{1}([0,T],G)$ and $u$ is a curve in the vector space $\mathfrak{g}$. More precisely, if $\mathfrak{g} = \vecspan\{e_{1},\hdots,e_{m},e_{m+1},\hdots,e_{n}\}$, then $u$ is given by
\begin{align}
u(t) = e_{0}+\sum_{i=1}^{m}u^{i}(t)e_{i},\nonumber
\end{align}
where $e_{0}\in\mathfrak{g}$ and the $m$-tuple of control inputs $[u^{1}\hdots u^{m}]^{T}$ take values in $\mathbb{R}^{m}$.
\end{definition}

\begin{remark}
If $m<n$, then the left-invariant control system is under-actuated otherwise it is fully-actuated.
\end{remark}

In what follows, we will fix $\mathfrak{k} = \vecspan\{e_{1},\hdots,e_{m}\}$, $\mathfrak{p} = \vecspan\{e_{m+1},\hdots,e_{n}\}$, with $e_{0}\in\mathfrak{p}$ and $E\colonequals G\times\mathfrak{k}$. Consider the following OCP
\begin{align}
&\min_{u(\cdot)}J\colonequals\min_{u(\cdot)}\int_0^T[C(g(t),u(t))+V(g(t))]dt\nonumber
\intertext{subject to\hfill\textbf{(P1)}}
&\dot{g}(t) = T_{e}L_{g(t)}(u(t)), \hspace{5pt} g(0) = g_{0}, \hspace{5pt} g(T) = g_{T},\nonumber
\end{align}
where the following assumptions hold:\\

\begin{enumerate}[(i)]
\item{$C: E\to\mathbb{R}$ is a $G$-invariant function (under a suitable left action of $G$ on $E$, which will be defined shortly) and is also sufficiently regular.\label{a1}}
\item{$V: G\to\mathbb{R}$ (potential function) is not a $G$-invariant function and is also sufficiently regular.}
\item{The potential function also depends on a parameter $\alpha_{0}$ that may be considered to be an element of the dual space $X^{*}$ of a vector space $X$. Hence, we define the extended potential function as $V_{\textnormal{ext}}: G\times X^{*}\to\mathbb{R}$, with $V_{\textnormal{ext}}(\cdot,\alpha_{0}) = V$.}
\item{There is a left representation $\rho$ of $G$ on $X$, i.e., $\rho_{\cdot}: G\to\mathrm{GL}(X)$ is a homomorphism. Hence, there is a left representation $\rho^{*}$ of $G$ on $X^{*}$, i.e., $\rho_{\cdot}^{*}: G\to\mathrm{GL}(X^{*})$, defined as the adjoint of $\rho$, i.e., for any $g \in G$, we define $\rho_{g^{-1}}^{*}\in\mathrm{GL}(X^{*})$ as the adjoint of $\rho_{g^{-1}}\in\mathrm{GL}(X)$, i.e., for any $x\in X$ and $\alpha \in X^{*}$
\begin{align}
\langle\rho^{*}_{g^{-1}}(\alpha),x\rangle = \langle\alpha,\rho_{g^{-1}}(x)\rangle.\nonumber
\end{align}
As a result, there is a left action of $G$ on $G \times X^{*}$ defined as
\begin{align}
\Phi: G \times (G \times X^{*})&\longrightarrow G\times X^{*},\nonumber\\
(g,(h,\alpha))&\longmapsto(L_{g}(h),\rho_{g^{-1}}^{*}(\alpha)).\label{eq_phi}
\end{align}}
\item{The extended potential function is $G$-invariant under \eqref{eq_phi}, i.e., $V_{\textnormal{ext}}\circ\Phi_{g} = V_{\textnormal{ext}}$, for any $g \in G$, or more concretely $V_{\textnormal{ext}}(L_{g}(h),\rho^{*}_{g^{-1}}(\alpha)) = V_{\textnormal{ext}}(h,\alpha)$, for any $h\in G$ and $\alpha\in X^{*}$.\label{a5}}
\item{The potential function is invariant under the left action of the isotropy group
\begin{align}
G_{\alpha_{0}}\colonequals\{g\in G\mid\rho_{g}^{*}(\alpha_{0}) = \alpha_{0}\}.\nonumber
\end{align}\label{a6}}
\end{enumerate}

\begin{remark}
Note that $E$ is a trivial vector bundle over $G$ and define the left action of $G$ on $E$ as follows
\begin{align}
\Psi: G\times E&\longrightarrow E,\nonumber\\
(g,(h,u))&\longmapsto(L_{g}(h),u).\label{eq_psi}
\end{align}
Throughout the paper, we assume that $C: E\to\mathbb{R}$ is $G$-invariant under \eqref{eq_psi}, i.e., $C\circ\Psi_{g} = C$, for any $g\in G$.
\end{remark}

\subsection{Euler--Poincar{\'e} Reduction}
We can solve $\textbf{(P1)}$ as a constrained variational problem using the method of Lagrange multipliers (see, e.g, \cite{Bloch2015}, \cite{Koon1997}). Define the augmented Lagrangian $L_{a}: E\times\mathfrak{p}^{*}\times X^{*}\to\mathbb{R}$ as follows
\begin{align}
L_{a}(g,u,\lambda,\alpha) = C(g,u-e_{0})+V_{\textnormal{ext}}(g,\alpha)+\langle\lambda,u-e_{0}\rangle,\nonumber
\end{align}
where $\lambda(\cdot)\in C^{1}([0,T],\mathfrak{p}^{*})$. Define the left action of $G$ on $E\times\mathfrak{p}^{*}\times X^{*}$ as follows
\begin{align}
\Upsilon: G\times (E\times\mathfrak{p}^{*}\times X^{*})&\longrightarrow E\times\mathfrak{p}^{*}\times X^{*},\nonumber\\
(g,(h,u,\lambda,\alpha))&\longmapsto(L_{g}(h),u,\lambda,\rho_{g^{-1}}^{*}(\alpha)).\label{eq_upsilon}
\end{align}
Under the assumption \eqref{a5}, it follows that $L_{a}: E\times\mathfrak{p}^{*}\times X^{*}\to\mathbb{R}$ is $G$-invariant under \eqref{eq_upsilon}, i.e., $L_{a}\circ\Upsilon_{g} = L_{a}$, for any $g\in G$. In particular, under the assumption \eqref{a6}, it follows that the augmented Lagrangian $L_{a}(\cdot,\cdot,\cdot,\alpha_{0})\equalscolon L_{a,\alpha_{0}}: E\times\mathfrak{p}^{*}\to\mathbb{R}$ is $G_{\alpha_{0}}$-invariant under \eqref{eq_upsilon}, i.e., $L_{a,\alpha_{0}}\circ\Upsilon_{g} = L_{a,\alpha_{0}}$, for any $g\in G_{\alpha_{0}}$. We can now obtain the reduced augmented Lagrangian $\ell_{a}: \mathfrak{k}\times\mathfrak{p}^{*}\times X^{*}\to\mathbb{R}$, which is given by
\begin{align}
\ell_{a}(u,\lambda,\alpha) &\colonequals L_{a}(e,u,\lambda,\alpha)\nonumber\\
&= C(u-e_{0})+V_{\textnormal{ext}}(\alpha)+\langle\lambda,u-e_{0}\rangle,\nonumber
\end{align}
where $\alpha = \rho^{*}_{g}(\alpha_{0})$ and with a slight abuse of notation, we write $C(e,u-e_{0}) = C(u-e_{0})$ and $V_{\textnormal{ext}}(\alpha) = V_{\textnormal{ext}}(e,\alpha)$. A normal extremal for $\textbf{(P1)}$ satisfies the following Euler--Poincar{\'e} equations (see Theorem \ref{thm4} below for a proof)
\begin{align}
\frac{d}{dt}\left(\frac{\partial C}{\partial u}+\lambda\right) &= \ad^{*}_{u}\left(\frac{\partial C}{\partial u}+\lambda\right)+\mathbf{J}_{X}\left(\frac{\partial V_{\textnormal{ext}}}{\partial\alpha},\alpha\right),\label{eq_ep}\\
\dot{\alpha} &= \rho_{u}^{\prime*}(\alpha), \hspace{5pt} \alpha(0) = \rho^{*}_{g_{0}}(\alpha_{0}),\label{eq_diff_ep}
\end{align}
where $\mathbf{J}_{X}: T^{*}X\cong X\times X^{*}\to\mathfrak{g}^{*}$ is the momentum map corresponding to the left action of $G$ on $X$ defined using the left representation $\rho$ of $G$ on $X$, i.e., for any $x$, $\xi\in X$ and $\alpha\in X^{*}$
\begin{align}
\langle\mathbf{J}_{X}(x,\alpha),\xi\rangle = \langle\alpha,\xi_{X}(x)\rangle,\label{eq_mp}
\end{align}
with $\xi_{X}$ being the infinitesimal generator of the left action of $G$ on $X$ and $\rho_{\cdot}^{\prime*}: \mathfrak{g}\to\mathfrak{gl}(X^{*})$ is defined as the adjoint of $\rho^{\prime}$, which is the representation induced by the left representation $\rho$ of $G$ on $X$. Note that the solution to \eqref{eq_diff_ep} is given by $\alpha(\cdot) = \rho^{*}_{g(\cdot)}(\alpha_{0})$. Also, note that using the notation of \cite{Holm1998}, we have $\mathbf{J}_{X}(x,\alpha) = x\diamond\alpha$. For more details, see \cite{Cendra2001}, \cite{Holm1998}, \cite{Holm2009}. To summarize, we have the following theorem.

\begin{theorem}\label{thm4}
A normal extremal for $\textnormal{\textbf{(P1)}}$, under the assumptions \eqref{a1}--\eqref{a6}, satisfies the following Euler--Poincar{\'e} equations
\begin{align}
\frac{d}{dt}\left(\frac{\partial C}{\partial u}+\lambda\right) &= \ad^{*}_{u}\left(\frac{\partial C}{\partial u}+\lambda\right)+\mathbf{J}_{X}\left(\frac{\partial V_{\textnormal{ext}}}{\partial\alpha},\alpha\right),\label{eq_ep_thm4}\\
\dot{\alpha} &= \rho_{u}^{\prime*}(\alpha), \hspace{5pt} \alpha(0) = \rho^{*}_{g_{0}}(\alpha_{0}).\label{eq_diff_ep_thm4}
\end{align}
\end{theorem}

\begin{proof}
The proof follows arguments similar to the ones given in \cite{Holm1998}, \cite{Koon1997}. Consider the following variational principles:\\

\begin{enumerate}[(a)]
\item{The variational principle 
\begin{align}
\delta\int_{0}^{T}L_{a,\alpha_{0}}(g(t),u(t),\lambda(t))dt = 0,\nonumber
\end{align}
holds for all variations of $g$ (vanishing at the endpoints) and $u$.\label{v1_thm4}}
\item{The constrained variational principle
\begin{align}
\delta\int_{0}^{T}\ell_{a}(u(t),\lambda(t),\alpha(t))dt = 0,\nonumber
\end{align}
holds using variations of $u$ and $\alpha$ of the form
\begin{align}
\delta{u} &= \dot{\eta}+\ad_{u}\eta,\nonumber\\
\delta{\alpha} &= \rho_{\eta}^{\prime*}(\alpha),\nonumber
\end{align}
where $\eta(\cdot)\in C^{1}([0,T],\mathfrak{g})$ vanishes at the endpoints.\label{v2_thm4}}
\end{enumerate}
\vspace{12pt}
We will first show that the variational principle \eqref{v1_thm4} implies the constrained variational principle \eqref{v2_thm4}. We begin by noting that $L_{a}: E\times\mathfrak{p}^{*}\times X^{*}\to\mathbb{R}$ is $G$-invariant under \eqref{eq_upsilon}, i.e., $L_{a}\circ\Upsilon_{g} = L_{a}$, for any $g\in G$ and $\alpha = \rho^{*}_{g}(\alpha_{0})$, so the integrand in the variational principle \eqref{v1_thm4} is equal to the integrand in the constrained variational principle \eqref{v2_thm4}. However, all variations of $g$ vanishing at the endpoints induce and are induced by variations of $u$ of the form $\delta{u} = \dot{\eta}+\ad_{u}\eta$, with $\eta(0) = \eta(T) = 0$. The relation between $\delta{g}$ and $\eta$ is given by $\eta = T_{g}L_{g^{-1}}(\delta{g})$ (see Proposition 5.1 in \cite{Bloch1996}, which is Lemma 3.2 in \cite{Holm1998}). So, if the variational principle \eqref{v1_thm4} holds and if we define $\eta = T_{g}L_{g^{-1}}(\delta{g})$ and $\delta{u} = T_{g}L_{g^{-1}}(\dot{g})$, then by Proposition 5.1 in \cite{Bloch1996}, we have $\delta{u} = \dot{\eta}+\ad_{u}\eta$. In addition, we have $\delta{\alpha} = \rho_{\eta}^{\prime*}(\alpha)$. Hence, the variational principle \eqref{v1_thm4} implies the constrained variational principle \eqref{v2_thm4}.

A normal extremal for $\textbf{(P1)}$ satisfies the variational principle \eqref{v1_thm4} and hence, the constrained variational principle \eqref{v2_thm4}, by the above discussion. So, we have
\begin{align}
0 &= \delta\int_{0}^{T}\ell_{a}(u(t),\lambda(t),\alpha(t))dt\nonumber\\
&= \int_{0}^{T}\left[\left\langle\frac{\partial C}{\partial u},\delta{u}\right\rangle+\langle\lambda,\delta{u}\rangle+\left\langle\delta{\alpha},\frac{\partial V_{\textnormal{ext}}}{\partial\alpha}\right\rangle\right]dt\nonumber\\
&= \int_{0}^{T}\left[\left\langle\frac{\partial C}{\partial u}+\lambda,\delta{u}\right\rangle+\left\langle\delta{\alpha},\frac{\partial V_{\textnormal{ext}}}{\partial\alpha}\right\rangle\right]dt\nonumber\\
&= \int_{0}^{T}\left[\left\langle\frac{\partial C}{\partial u}+\lambda,\dot{\eta}+\ad_{u}\eta\right\rangle+\left\langle\delta{\alpha},\frac{\partial V_{\textnormal{ext}}}{\partial\alpha}\right\rangle\right]dt\nonumber\\
&= \int_{0}^{T}\left[\left\langle-\frac{d}{dt}\left(\frac{\partial C}{\partial u}+\lambda\right)+\ad_{u}^{*}\left(\frac{\partial C}{\partial u}+\lambda\right),\eta\right\rangle+\left\langle\rho_{\eta}^{\prime*}(\alpha),\frac{\partial V_{\textnormal{ext}}}{\partial\alpha}\right\rangle\right]dt\nonumber\\
&= \int_{0}^{T}\left[\left\langle-\frac{d}{dt}\left(\frac{\partial C}{\partial u}+\lambda\right)+\ad_{u}^{*}\left(\frac{\partial C}{\partial u}+\lambda\right)+\mathbf{J}_{X}\left(\frac{\partial V_{\textnormal{ext}}}{\partial\alpha},\alpha\right),\eta\right\rangle\right]dt,\nonumber
\end{align}
where we have used integration by parts along with the fact that $\eta(0) = \eta(T) = 0$ and so, we have
\begin{align}
\frac{d}{dt}\left(\frac{\partial C}{\partial u}+\lambda\right) = \ad^{*}_{u}\left(\frac{\partial C}{\partial u}+\lambda\right)+\mathbf{J}_{X}\left(\frac{\partial V_{\textnormal{ext}}}{\partial\alpha},\alpha\right).\nonumber
\end{align}
Finally, by taking the time derivative of $\alpha$, we have
\begin{align}
\dot{\alpha} = \rho_{u}^{\prime*}(\alpha), \hspace{5pt} \alpha(0) = \rho^{*}_{g_{0}}(\alpha_{0}).\nonumber
\end{align}
\end{proof}

There are two special cases of Theorem \ref{thm4} that are of particular interest and we state them as corollaries.

\begin{corollary}\label{cor5}
Let $X = \mathfrak{g}$ and $\rho$ be the adjoint representation of $G$ on $X$, i.e., $\rho_{g} = \Ad_{g}$, for any $g\in G$. Then, the Euler--Poincar{\'e} equations \eqref{eq_ep_thm4}--\eqref{eq_diff_ep_thm4} give the following equations
\begin{align}
\frac{d}{dt}\left(\frac{\partial C}{\partial u}+\lambda\right) &= \ad^{*}_{u}\left(\frac{\partial C}{\partial u}+\lambda\right)-\ad_{\frac{\partial V_{\textnormal{ext}}}{\partial\alpha}}^{*}\alpha,\nonumber\\
\dot{\alpha} &= \ad_{u}^{*}\alpha, \hspace{5pt} \alpha(0) = \Ad_{g_{0}}^{*}\alpha_{0}.\nonumber
\end{align}
\end{corollary}

\begin{proof}
We first begin by noting that $\alpha\in X^{*} = \mathfrak{g}^{*}$ and $\rho^{*}$ is the coadjoint representation of $G$ on $X^{*}$, i.e., $\rho^{*}_{g} = \Ad_{g}^{*}$, for any $g\in G$. We also have $\rho_{\xi}^{\prime} = \ad_{\xi}$, for any $\xi\in\mathfrak{g}$ and it follows that $\xi_{X}(x) = \ad_{\xi}x$, for any $x\in X$. From \eqref{eq_mp}, we have 
\begin{align}
\langle\mathbf{J}_{X}(x,\alpha),\xi\rangle &= \langle\alpha,\ad_{\xi}x\rangle\nonumber\\
&= \langle\alpha,-\ad_{x}\xi\rangle\nonumber\\
&= \langle-\ad_{x}^{*}\alpha,\xi\rangle,\nonumber
\end{align}
which gives $\mathbf{J}_{X}(x,\alpha) = -\ad_{x}^{*}\alpha$.
\end{proof}

\begin{corollary}\label{cor6}
Let $X = \mathfrak{g}^{*}$ and $\rho$ be the coadjoint representation of $G$ on $X$, i.e., $\rho_{g} = \Ad_{g^{-1}}^{*}$, for any $g\in G$. Then, the Euler--Poincar{\'e} equations \eqref{eq_ep_thm4}--\eqref{eq_diff_ep_thm4} give the following equations
\begin{align}
\frac{d}{dt}\left(\frac{\partial C}{\partial u}+\lambda\right) &= \ad^{*}_{u}\left(\frac{\partial C}{\partial u}+\lambda\right)+\ad_{\alpha}^{*}\frac{\partial V_{\textnormal{ext}}}{\partial\alpha},\nonumber\\
\dot{\alpha} &= -\ad_{u}\alpha, \hspace{5pt} \alpha(0) = \Ad_{g_{0}^{-1}}\alpha_{0}.\nonumber
\end{align}
\end{corollary}

\begin{proof}
We first begin by noting that $\alpha\in X^{*} = \mathfrak{g}^{**}\cong\mathfrak{g}$ and $\rho^{*}$ is the adjoint representation of $G$ on $X^{*}$, i.e., $\rho^{*}_{g} = \Ad_{g^{-1}}$, for any $g\in G$. We also have $\rho_{\xi}^{\prime} = -\ad_{\xi}^{*}$, for any $\xi\in\mathfrak{g}$ and it follows that $\xi_{X}(x) = -\ad_{\xi}^{*}x$, for any $x\in X$. From \eqref{eq_mp}, we have
\begin{align}
\langle\mathbf{J}_{X}(x,\alpha),\xi\rangle &= \langle\alpha,-\ad^{*}_{\xi}x\rangle\nonumber\\
&= \langle-\ad_{\xi}\alpha,x\rangle\nonumber\\
&= \langle x,\ad_{\alpha}\xi\rangle\nonumber\\
&= \langle\ad^{*}_{\alpha}x,\xi\rangle,\nonumber
\end{align}
which gives $\mathbf{J}_{X}(x,\alpha) = \ad^{*}_{\alpha}x$.
\end{proof}

The Euler--Poincar\'e equation \eqref{eq_ep_thm4} is not particularly feasible for describing the time evolution of $u$ and $\lambda$  because the equations for them are combined into a single equation. However, assuming an additional structure on the Lie algebra $\mathfrak{g}$, we may decouple the equations for $u$ and $\lambda$ (see \cite{Bloch1994} for a similar approach applied to the standard Euler--Poincar{\'e} equation):

\begin{proposition}\label{prop7}
Assume that $\mathfrak{g} = \mathfrak{k}\oplus\mathfrak{p}$ such that 
\begin{align}
[\mathfrak{k},\mathfrak{k}]\subseteq\mathfrak{p}, \hspace{5pt} [\mathfrak{p},\mathfrak{k}]\subseteq\mathfrak{k}, \hspace{5pt} [\mathfrak{p},\mathfrak{p}]\subseteq\mathfrak{p},\nonumber
\end{align}
then the time evolution of $u$ and $\lambda$ in \eqref{eq_ep_thm4} are given by the following equations 
\begin{align}
\frac{d}{dt}\frac{\partial C}{\partial u} &= \ad_{e_0}^{*}\frac{\partial C}{\partial u}+\ad_{u-e_{0}}^{*}\lambda+\mathbf{J}_{X}\left(\frac{\partial V_{\textnormal{ext}}}{\partial\alpha},\alpha\right)\bigg|_{\mathfrak{k}},\nonumber\\
\frac{d\lambda}{dt} &= \ad_{e_0}^{*}\lambda+\ad_{u-e_{0}}^{*}\frac{\partial C}{\partial u}+\mathbf{J}_{X}\left(\frac{\partial V_{\textnormal{ext}}}{\partial\alpha},\alpha\right)\bigg|_{\mathfrak{p}},\nonumber
\end{align}
where $u-e_{0}\in\mathfrak{k}$.
\end{proposition}

\begin{proof}
It is easy to verify that $\mathfrak{g}^{*} = \mathfrak{k}^{*}\oplus\mathfrak{p}^{*}$ such that 
\begin{align}
\ad_{\mathfrak{k}}^{*}\mathfrak{k}^{*}\subseteq\mathfrak{p}^{*}, \hspace{5pt} \ad_{\mathfrak{p}}^{*}\mathfrak{k}^{*}\subseteq\mathfrak{k}^{*}, \hspace{5pt} \ad_{\mathfrak{p}}^{*}\mathfrak{p}^{*}\subseteq\mathfrak{p}^{*}, \hspace{5pt} \ad_{\mathfrak{k}}^{*}\mathfrak{p}^{*}\subseteq\mathfrak{k}^{*}.\label{eq_rel}
\end{align}
It is also easy to verify that $\frac{\partial C}{\partial u}\in\mathfrak{k}^{*}$ and by definition $\lambda\in\mathfrak{p}^{*}$. We now have a splitting of the left hand side of \eqref{eq_ep_thm4} in $\mathfrak{k}^{*}$ and $\mathfrak{p}^{*}$. We have the following
\begin{align}
\ad_{u}^{*}\frac{\partial C}{\partial u} &= \ad_{e_{0}}^{*}\frac{\partial C}{\partial u}+\ad_{u-e_{0}}^{*}\frac{\partial C}{\partial u},\nonumber\\
\ad_{u}^{*}\lambda &= \ad_{e_{0}}^{*}\lambda+\ad_{u-e_{0}}^{*}\lambda,\nonumber
\end{align}
where $u-e_{0}\in\mathfrak{k}$. By using \eqref{eq_rel}, we have $\ad_{e_0}^{*}\frac{\partial C}{\partial u}\in\mathfrak{k}^{*}$, $\ad_{u-e_{0}}^{*}\frac{\partial C}{\partial u}\in\mathfrak{p}^{*}$, $\ad_{e_{0}}^{*}\lambda\in\mathfrak{p}^{*}$ and $\ad_{u-e_{0}}^{*}\lambda\in\mathfrak{k}^{*}$. We can also split $\mathbf{J}_{X}(\frac{\partial V_{\textnormal{ext}}}{\partial\alpha},\alpha)\in\mathfrak{g}^{*}$ in $\mathfrak{k}^{*}$ and $\mathfrak{p}^{*}$. We now have a splitting of the right hand side of \eqref{eq_ep_thm4} in $\mathfrak{k}^{*}$ and $\mathfrak{p}^{*}$. So, \eqref{eq_ep_thm4} splits into the following equations
\begin{align}
\frac{d}{dt}\frac{\partial C}{\partial u} &= \ad_{e_0}^{*}\frac{\partial C}{\partial u}+\ad_{u-e_{0}}^{*}\lambda+\mathbf{J}_{X}\left(\frac{\partial V_{\textnormal{ext}}}{\partial\alpha},\alpha\right)\bigg|_{\mathfrak{k}},\nonumber\\
\frac{d\lambda}{dt} &= \ad_{e_0}^{*}\lambda+\ad_{u-e_{0}}^{*}\frac{\partial C}{\partial u}+\mathbf{J}_{X}\left(\frac{\partial V_{\textnormal{ext}}}{\partial\alpha},\alpha\right)\bigg|_{\mathfrak{p}}.\nonumber
\end{align}
\end{proof}

\begin{remark}
Note that semisimple Lie algebras admit a Cartan decomposition, i.e., if $\mathfrak{g}$ is semisimple, then $\mathfrak{g} = \mathfrak{k}\oplus\mathfrak{p}$ such that 
\begin{align}
[\mathfrak{k},\mathfrak{k}]\subseteq\mathfrak{p}, \hspace{5pt} [\mathfrak{p},\mathfrak{k}]\subseteq\mathfrak{k}, \hspace{5pt} [\mathfrak{p},\mathfrak{p}]\subseteq\mathfrak{p},\nonumber
\end{align}
where $\mathfrak{k} = \{x\in\mathfrak{g}\mid\theta(x) = -x\}$ is the $-1$ eigenspace of the Cartan involution $\theta$ and $\mathfrak{p} = \{x\in\mathfrak{g}\mid\theta(x) = x\}$ is the $+1$ eigenspace of the Cartan involution $\theta$. In addition, the Killing form is positive definite on $\mathfrak{k}$ and negative definite on $\mathfrak{p}$. So, connected semisimple Lie groups are potential candidates that satisfy the assumption of Proposition \ref{prop7}. Conversely, a Cartan decomposition determines a Cartan involution $\theta$ (see, e.g., \cite{Knapp2002}). For more details, see \cite{D'Alessandro2007}, \cite{Helgason2001}, \cite{Knapp2002}. Also, note that the roles of $\mathfrak{k}$ and $\mathfrak{p}$ can be reversed in Proposition \ref{prop7}.
\end{remark}

\subsection{Legendre Transformation}
If we assume that the reduced Lagrangian $\ell_{a}$ is hyper-regular, then we can obtain the reduced Hamiltonian $h_{a}: \mathfrak{g}^{*}\times\mathfrak{p}^{*}\times X^{*}\to\mathbb{R}$ (by using a Legendre transformation) given by
\begin{align}
h_{a}(\mu,\lambda,\alpha) = \langle\mu,u\rangle-\ell_{a}(u,\lambda,\alpha),\nonumber
\end{align}
where $\mu = \frac{\partial\ell_{a}}{\partial u} = (\frac{\partial C}{\partial u}+\lambda)$, with $\mu(\cdot)\in C^{1}([0,T],\mathfrak{g}^{*})$. The Euler--Poincar{\'e} equations \eqref{eq_ep_thm4}--\eqref{eq_diff_ep_thm4} can now be written as the Lie--Poisson equations (see, e.g., \cite{Balmaz2010}, \cite{Holm1998}), which are given by
\begin{align}
\dot{\mu} &= \ad^{*}_{u}\mu+\mathbf{J}_{X}\left(\frac{\partial V_{\textnormal{ext}}}{\partial\alpha},\alpha\right),\label{eq_lp}\\
\dot{\alpha} &= \rho_{u}^{\prime*}(\alpha), \hspace{5pt} \alpha(0) = \rho^{*}_{g_{0}}(\alpha_{0}).\label{eq_diff_lp}
\end{align}

\begin{remark}
The Lie--Poisson equations \eqref{eq_lp}--\eqref{eq_diff_lp} are also obtained in \cite{Borum2015}, \cite{Borum2014} using Pontryagin's maximum principle. For more details, see \cite{Borum2015}, \cite{Borum2014}.
\end{remark}

\subsection{Discrete Lagrange--Pontryagin Principle}
Consider the discrete-time version of\\ $\textbf{(P1)}$ given by
\begin{align}
&\min_{\{u_{k}\}_{k=0}^{N-1}}J_{d}\colonequals\min_{\{u_{k}\}_{k=0}^{N-1}}\sum_{k=0}^{N-1}[C_{d}(g_{k},u_{k})+V_{d}(g_{k})]\nonumber
\intertext{subject to\hfill\textbf{(P2)}}
&g_{k+1} = g_{k}\tau(hu_{k}), \hspace{5pt} \text{with given boundary conditions} \hspace{5pt} g_{0} \hspace{5pt} \text{and} \hspace{5pt} g_{N},\nonumber
\end{align}
where $h\in\mathbb{R}_{>0}$ is the time-step, $\tau: \mathfrak{g}\to G$ is the retraction map (see, e.g., \cite{BouRabee2007}, \cite{BouRabee2009}, \cite{Kobilarov2009}, \cite{Kobilarov2007}, \cite{Kobilarov2011}) and the functions $C_{d}$ and $V_{d}$ satisfy assumptions \eqref{a1}--\eqref{a6}.

We can solve $\textbf{(P2)}$ using a discrete analogue of the Lagrange--Pontryagin variational principle. Define the augmented cost function as follows 
\begin{align}
J_{d,a} = \sum_{k=0}^{N-1} L_{d,a}(g_{k},g_{k+1},u_{k},\mu_{k},\alpha),\nonumber
\end{align}
where the augmented Lagrangian $L_{d,a}: G\times E\times\mathfrak{g}^{*}\times X^{*}\to\mathbb{R}$ is defined as follows
\begin{align}
L_{d,a}(g_{k},g_{k+1},u_{k},\mu_{k},\alpha) &= C_{d}(g_{k},u_{k}-e_{0})+V_{d,\textnormal{ext}}(g_{k},\alpha)\\
&\quad+h\left\langle\mu_k,\frac{1}{h}\tau^{-1}(g_k^{-1}g_{k+1})-(u_k-e_{0})\right\rangle.\nonumber
\end{align}
Define the left action of $G$ on $G\times E\times\mathfrak{g}^{*}\times X^{*}$ as follows
\begin{align}
\Gamma: G\times (G\times E\times\mathfrak{g}^{*}\times X^{*})&\longrightarrow G\times E\times\mathfrak{g}^{*}\times X^{*},\nonumber\\
(g,(h_{1},h_{2},u,\mu,\alpha))&\longmapsto(L_{g}(h_{1}),L_{g}(h_{2}),u,\mu,\rho_{g^{-1}}^{*}(\alpha)).\label{eq_gamma}
\end{align}
Under the assumption \eqref{a5}, it follows that $L_{d,a}: G\times E\times\mathfrak{g}^{*}\times X^{*}\to\mathbb{R}$ is $G$-invariant under \eqref{eq_gamma}, i.e., $L_{d,a}\circ\Gamma_{g} = L_{d,a}$, for any $g\in G$. In particular, under the assumption \eqref{a6}, it follows that the augmented Lagrangian $L_{d,a}(\cdot,\cdot,\cdot,\cdot,\alpha_{0})\equalscolon L_{d,a,\alpha_{0}}: G\times E\times\mathfrak{g}^{*}\to\mathbb{R}$ is $G_{\alpha_{0}}$-invariant under \eqref{eq_gamma}, i.e., $L_{d,a,\alpha_{0}}\circ\Gamma_{g} = L_{d,a,\alpha_{0}}$, for any $g\in G_{\alpha_{0}}$. We can now obtain the reduced augmented Lagrangian $\ell_{d,a}: E\times\mathfrak{g}^{*}\times X^{*}\to\mathbb{R}$, which is given by
\begin{align}
\ell_{d,a}(\pi(g_{k},g_{k+1}),u_{k},\mu_{k},\bar{\alpha}_{k})&\colonequals L_{d,a}(e,\pi(g_{k},g_{k+1}),u_{k},\mu_{k},\bar{\alpha}_{k}),\nonumber\\
&= C_{d}(u_{k}-e_{0})+V_{d,\textnormal{ext}}(\bar{\alpha}_{k})\nonumber\\
&\quad+h\left\langle\mu_k,\frac{1}{h}\tau^{-1}(\pi(g_{k},g_{k+1}))-(u_k-e_{0})\right\rangle,\nonumber
\end{align}
where $\pi(g_{k},g_{k+1})\colonequals g_{k}^{-1}g_{k+1}\in G$, $\bar{\alpha}_{k} = \rho^{*}_{g_{k}}(\alpha_{0})$ and with a slight abuse of notation, we write $C_{d}(e,u_{k}-e_{0}) = C_{d}(u_{k}-e_{0})$ and $V_{d,\textnormal{ext}}(\bar{\alpha}_{k}) = V_{d,\textnormal{ext}}(e,\bar{\alpha}_{k})$. A normal extremal for $\textbf{(P2)}$ satisfies the following discrete-time Lie--Poisson equations (see Theorem \ref{thm11} below for a proof)
\begin{align}
(d\tau_{hu_{k}}^{-1})^{*}\mu_{k} &= (d\tau_{-hu_{k-1}}^{-1})^{*}\mu_{k-1}+\mathbf{J}_{X}\left(\frac{\partial V_{d,\textnormal{ext}}}{\partial\bar{\alpha}_{k}},\bar{\alpha}_{k}\right),\label{eq_dlp}\\
\bar{\alpha}_{k+1} &= \rho_{\tau(hu_{k})}^{*}(\bar{\alpha}_{k}), \hspace{5pt} \bar{\alpha}_{0} = \rho^{*}_{g_{0}}(\alpha_{0})\label{eq_diff_dlp},
\end{align}
where $\mu_{k} = \frac{1}{h}\frac{\partial C_{d}}{\partial u_{k}}\in\mathfrak{g}^{*}$. To summarize, we have the following theorem but before proceeding further, we recall a lemma, which will be used in its proof.

\begin{lemma}[\cite{BouRabee2007}, \cite{BouRabee2009}, \cite{Kobilarov2011}]\label{lem10}
The following properties hold:
\begin{align}
\frac{1}{h}\delta\tau^{-1}(g_{k}^{-1}g_{k+1}) = \frac{1}{h}\delta\tau^{-1}(\pi(g_{k},g_{k+1})) = \frac{1}{h}d\tau_{hu_{k}}^{-1}(-\eta_{k}+\Ad_{\tau(hu_k)}\eta_{k+1}),
\end{align}
where $\eta_{k} = T_{g_{k}}L_{g_{k}^{-1}}(\delta{g_{k}})\in\mathfrak{g}$.
\begin{align}
(d\tau_{-hu_{k}}^{-1})^{*}\mu_{k} = \Ad_{\tau(hu_{k})}^{*}(d\tau_{hu_{k}}^{-1})^{*}\mu_{k},
\end{align}
where $\mu_{k}\in\mathfrak{g}^{*}$ and $d\tau^{-1}$ is the inverse right trivialized tangent of $\tau$.
\end{lemma}

\begin{theorem}\label{thm11}
A normal extremal for $\textnormal{\textbf{(P2)}}$, under the assumptions \eqref{a1}--\eqref{a6}, satisfies the following discrete-time Lie--Poisson equations
\begin{align}
(d\tau_{hu_{k}}^{-1})^{*}\mu_{k} &= (d\tau_{-hu_{k-1}}^{-1})^{*}\mu_{k-1}+\mathbf{J}_{X}\left(\frac{\partial V_{d,\textnormal{ext}}}{\partial\bar{\alpha}_{k}},\bar{\alpha}_{k}\right),\label{eq_dlp_thm11}\\
\bar{\alpha}_{k+1} &= \rho_{\tau(hu_{k})}^{*}(\bar{\alpha}_{k}), \hspace{5pt} \bar{\alpha}_{0} = \rho^{*}_{g_{0}}(\alpha_{0})\label{eq_diff_dlp_thm11},
\end{align}
where $\mu_{k} = \frac{1}{h}\frac{\partial C_{d}}{\partial u_{k}}\in\mathfrak{g}^{*}$.
\end{theorem}

\begin{proof}
Consider the following variational principles:\\

\begin{enumerate}[(a)]
\item{The variational principle 
\begin{align}
\delta\sum_{k=0}^{N-1} L_{d,a,\alpha_{0}}(g_{k},g_{k+1},u_{k},\mu_{k}) = 0,\nonumber
\end{align}
holds for all variations of $g_{k}$ (vanishing at the endpoints), $\tau^{-1}(g_{k}^{-1}g_{k+1})$ (induced by the variations of $g_{k}$) and $u_{k}$.\label{v1_thm11}}
\item{The variational principle
\begin{align}
\delta\sum_{k=0}^{N-1}\ell_{d,a}(\pi(g_{k},g_{k+1}),u_{k},\mu_{k},\bar{\alpha}_{k}) = 0,\nonumber
\end{align}
holds for all variations of $\tau^{-1}(\pi(g_{k},g_{k+1}))$ (induced by the variations of $g_{k}$ vanishing at the endpoints), $u_{k}$ and $\bar{\alpha}_{k}$ of the form
\begin{align}
\delta\bar{\alpha}_{k} = \rho_{\eta_{k}}^{\prime*}(\bar{\alpha}_{k}),\nonumber
\end{align}
where $\eta_{k}\in\mathfrak{g}$ vanishes at the endpoints.\label{v2_thm11}}
\end{enumerate}
\vspace{12pt} 
We will first show that the variational principle \eqref{v1_thm11} implies the variational principle \eqref{v2_thm11}. We begin by noting that $L_{d,a}: G\times E\times\mathfrak{g}^{*}\times X^{*}\to\mathbb{R}$ is $G$-invariant under \eqref{eq_gamma}, i.e., $L_{d,a}\circ\Gamma_{g} = L_{d,a}$, for any $g\in G$ and $\bar{\alpha}_{k} = \rho^{*}_{g_{k}}(\alpha_{0})$, so the summand in the variational principle \eqref{v1_thm11} is equal to the summand in the variational principle \eqref{v2_thm11}. In addition, we have $\delta\bar{\alpha}_{k} = \rho_{\eta_{k}}^{\prime*}(\bar{\alpha}_{k})$. Hence, the variational principle \eqref{v1_thm11} implies the variational principle \eqref{v2_thm11}.

A normal extremal for $\textbf{(P2)}$, satisfies the variational principle \eqref{v1_thm11} and hence, the variational principle \eqref{v2_thm11}. So, we have
\begin{align}
0 &= \delta\sum_{k=0}^{N-1}\ell_{d,a}(\pi(g_{k},g_{k+1}),u_{k},\mu_{k},\bar{\alpha}_{k})\nonumber\\
&= \sum_{k=0}^{N-1}\bigg[\left\langle\frac{\partial C_{d}}{\partial u_{k}},\delta u_{k}\right\rangle-h\langle\mu_{k},\delta u_{k}\rangle+\left\langle\delta{\bar{\alpha}_{k}},\frac{\partial V_{d,\textnormal{ext}}}{\partial\bar{\alpha}_{k}}\right\rangle\nonumber\\
&\quad+h\left\langle\mu_k,\frac{1}{h}d\tau^{-1}_{hu_k}(-\eta_{k}+\Ad_{\tau(hu_k)}\eta_{k+1})\right\rangle\bigg]\nonumber\\
&= \sum_{k=0}^{N-1}\bigg[\left\langle\frac{\partial C_{d}}{\partial u_{k}}-h\mu_{k},\delta u_{k}\right\rangle+\left\langle\delta{\bar{\alpha}_{k}},\frac{\partial V_{d,\textnormal{ext}}}{\partial\bar{\alpha}_{k}}\right\rangle\nonumber\\
&\quad+h\left\langle\mu_k,\frac{1}{h}d\tau^{-1}_{hu_k}(-\eta_{k}+\Ad_{\tau(hu_k)}\eta_{k+1})\right\rangle\bigg]\nonumber\\
&= \sum_{k=1}^{N-1}\bigg[\left\langle\frac{\partial C_{d}}{\partial u_{k}}-h\mu_{k},\delta u_{k}\right\rangle+\left\langle\rho_{\eta_{k}}^{\prime*}(\bar{\alpha}_{k}),\frac{\partial V_{d,\textnormal{ext}}}{\partial\bar{\alpha}_{k}}\right\rangle\nonumber\\
&\quad+\left\langle(d\tau_{-hu_{k-1}}^{-1})^{*}\mu_{k-1}-(d\tau_{hu_{k}}^{-1})^{*}\mu_{k},\eta_{k}\right\rangle\bigg]\nonumber\\
&= \sum_{k=1}^{N-1}\bigg[\left\langle\frac{\partial C_{d}}{\partial u_{k}}-h\mu_{k},\delta u_{k}\right\rangle\nonumber\\
&\quad+\left\langle(d\tau_{-hu_{k-1}}^{-1})^{*}\mu_{k-1}+\mathbf{J}_{X}\left(\frac{\partial V_{d,\textnormal{ext}}}{\partial\bar{\alpha}_{k}},\bar{\alpha}_{k}\right)-(d\tau_{hu_{k}}^{-1})^{*}\mu_{k},\eta_{k}\right\rangle\bigg],\nonumber
\end{align}
where we have used Lemma \ref{lem10}, the analogue of integration by parts in the discrete-time setting along with the fact that $\eta_{0} = \eta_{N} = 0$ and so, we have
\begin{align}
(d\tau_{hu_{k}}^{-1})^{*}\mu_{k} &= (d\tau_{-hu_{k-1}}^{-1})^{*}\mu_{k-1}+\mathbf{J}_{X}\left(\frac{\partial V_{d,\textnormal{ext}}}{\partial\bar{\alpha}_{k}},\bar{\alpha}_{k}\right).\nonumber
\end{align}
Finally, after a simple calculation, we have
\begin{align}
\bar{\alpha}_{k+1} = \rho_{\tau(hu_{k})}^{*}(\bar{\alpha}_{k}), \hspace{5pt} \bar{\alpha}_{0} = \rho^{*}_{g_{0}}(\alpha_{0})\nonumber.
\end{align}
\end{proof}

As in the continuous-time case, there are are two special cases of Theorem \ref{thm11} that are of particular interest and we state them as corollaries.

\begin{corollary}\label{cor12}
Let $X = \mathfrak{g}$ and $\rho$ be the adjoint representation of $G$ on $X$, i.e., $\rho_{g} = \Ad_{g}$, for any $g\in G$. Then, the discrete-time Lie--Poisson equations \eqref{eq_dlp_thm11}--\eqref{eq_diff_dlp_thm11} give the following equations
\begin{align}
(d\tau_{hu_{k}}^{-1})^{*}\mu_{k} &= (d\tau_{-hu_{k-1}}^{-1})^{*}\mu_{k-1}-\ad_{\frac{\partial V_{d,\textnormal{ext}}}{\partial\bar{\alpha}_{k}}}^{*}\bar{\alpha}_{k},\nonumber\\
\bar{\alpha}_{k+1} &= \Ad_{\tau(hu_{k})}^{*}\bar{\alpha}_{k}, \hspace{5pt} \bar{\alpha}_{0} = \Ad_{g_{0}}^{*}\alpha_{0}\nonumber,
\end{align}
where $\mu_{k} = \frac{1}{h}\frac{\partial C_{d}}{\partial u_{k}}\in\mathfrak{g}^{*}$.
\end{corollary}

\begin{proof}
See Corollary \ref{cor5}.
\end{proof}

\begin{corollary}\label{cor13}
Let $X = \mathfrak{g}^{*}$ and $\rho$ be the coadjoint representation of $G$ on $X$, i.e., $\rho_{g} = \Ad_{g^{-1}}^{*}$, for any $g\in G$. Then, the discrete-time Lie--Poisson equations \eqref{eq_dlp_thm11}--\eqref{eq_diff_dlp_thm11} give the following equations
\begin{align}
(d\tau_{hu_{k}}^{-1})^{*}\mu_{k} &= (d\tau_{-hu_{k-1}}^{-1})^{*}\mu_{k-1}+\ad_{\bar{\alpha}_{k}}^{*}\frac{\partial V_{d,\textnormal{ext}}}{\partial\bar{\alpha}_{k}},\nonumber\\
\bar{\alpha}_{k+1} &= \Ad_{\tau(hu_{k})^{-1}}\bar{\alpha}_{k}, \hspace{5pt} \bar{\alpha}_{0} = \Ad_{g_{0}^{-1}}\alpha_{0}\nonumber,
\end{align}
where $\mu_{k} = \frac{1}{h}\frac{\partial C_{d}}{\partial u_{k}}\in\mathfrak{g}^{*}$.
\end{corollary}

\begin{proof}
See Corollary \ref{cor6}.
\end{proof}

\section{Examples}\label{sec4}
We will now consider two examples.

\subsection{The Heavy Top as an Optimal Control Problem}
We will now show that the heavy top equations can also be obtained from OCPs for left-invariant control systems on Lie groups, with partial symmetry breaking cost functions.

We now consider $\textbf{(P1)}$, with $G = \mathrm{SO}(3)$, $e_{0} = 0_{3\times 3}$ and with the cost function given by
\begin{align}
C(g,u) &= \frac{1}{2}\langle u,\mathbb{I}u\rangle,\nonumber\\
V(g) &= -m\mathbf{g}l\langle\mathbf{e}_{3},g\boldsymbol{\chi}\rangle,\nonumber
\end{align}
where $\langle\xi,\xi\rangle\colonequals\tr(\xi^{T}\xi)$, for any $\xi\in\mathfrak{g} = \mathfrak{so}(3)$, $u(\cdot) = \sum_{i=1}^{3}u^{i}(\cdot)e_{i}$, with the elements of the basis of $\mathfrak{g}$ (which is semisimple) given by
\begin{align}
e_{1} = \begin{bmatrix}
        0 & \phantom{-}0 & \phantom{-}0 \\
        0 & \phantom{-}0 & -1 \\
        0 & \phantom{-}1 & \phantom{-}0
        \end{bmatrix}, \hspace{5pt} 
e_{2} = \begin{bmatrix}
        \phantom{-}0 & \phantom{-}0 & \phantom{-}1 \\
        \phantom{-}0 & \phantom{-}0 & \phantom{-}0 \\
        -1 & \phantom{-}0 & \phantom{-}0
      \end{bmatrix}, \hspace{5pt}
e_{3} = \begin{bmatrix}
        0 & -1 & \phantom{-}0 \\
        1 & \phantom{-}0 & \phantom{-}0 \\
        0 & \phantom{-}0 & \phantom{-}0
        \end{bmatrix},\nonumber
\end{align}
which satisfy
\begin{align}
[e_{1},e_{2}] = e_{3}, \hspace{5pt} [e_{2},e_{3}] = e_{1}, \hspace{5pt} [e_{3},e_{1}] = e_{2},\nonumber
\end{align}
$\mathbb{I}: \mathfrak{g}\to\mathfrak{g}^{*} = \mathfrak{so}(3)^{*}$ is the inertia tensor of the top\footnote{It is calculated with respect to the pivot, which is not in general the center of mass.}, $m$ is the mass of the body, $\mathbf{g}$ is the acceleration due to gravity, $\mathbf{e}_{3}$ is the vertical unit vector, $\boldsymbol{\chi}\in\mathbb{R}^{3}$ is the unit vector from the point of support to the direction of the body's center of mass (constant) in body coordinates and $l$ is the length of the line segment between these two points.

Under the dual pairing, where $\langle\alpha,\xi\rangle\colonequals\tr(\alpha\xi)$, for any $\xi\in\mathfrak{g}$ and $\alpha\in\mathfrak{g}^{*}$, the elements of the basis of $\mathfrak{g}^{*}$ are given by
\begin{align}
e^{1} = \begin{bmatrix}
        0 & \phantom{-}0 & \phantom{-}0\\[4pt]
        0 & \phantom{-}0 & \phantom{-}\dfrac{1}{2}\\[4pt]
        0 & -\dfrac{1}{2} & \phantom{-}0
        \end{bmatrix}, \hspace{5pt} 
e^{2} = \begin{bmatrix}
        \phantom{-}0 & \phantom{-}0 & -\dfrac{1}{2}\\[8pt]
        \phantom{-}0 & \phantom{-}0 & \phantom{-}0\\[4pt]
        \phantom{-}\dfrac{1}{2} & \phantom{-}0 & \phantom{-}0
        \end{bmatrix}, \hspace{5pt}
e^{3} = \begin{bmatrix}
        \phantom{-}0 & \phantom{-}\dfrac{1}{2} & \phantom{-}0\\[4pt]
        -\dfrac{1}{2} & \phantom{-}0 & \phantom{-}0\\[8pt]
        \phantom{-}0 & \phantom{-}0 & \phantom{-}0
        \end{bmatrix}.\nonumber
\end{align}
It is easy to verify that the potential function is $\mathrm{SO}(2)$-invariant but not $G$-invariant and so, the potential function breaks the symmetry partially. We will now show how the theory developed in the paper can be applied to the OCP under consideration.

Let $X = \mathfrak{g}$ and $\rho$ be the adjoint representation of $G$ on $X$, i.e., $\rho_{g} = \Ad_{g}$, for any $g\in G$. So, $\rho^{*}$ is the coadjoint representation of $G$ on $X^{*} = \mathfrak{g}^{*}$, i.e., $\rho^{*}_{g} = \Ad_{g}^{*}$, for any $g\in G$. Let the extended potential function $V_{\textnormal{ext}}: G\times X^{*}\to\mathbb{R}$ be given as follows
\begin{align}
V_{\textnormal{ext}}(g,\breve{\alpha}) = -m\mathbf{g}l\langle g^{-1}\alpha,\boldsymbol{\chi}\rangle,\nonumber
\end{align}
where $\alpha\in\mathbb{R}^{3}$ is identified with $\breve{\alpha}\in\mathfrak{g}^{*}$ (see, e.g., \cite[Section 5.3]{Holm2009}). If we set $\breve{\alpha}_{0} = -2e^{3}$ in the extended potential function, we recover our original potential function. It is easy to verify that the extended potential function is $G$-invariant under \eqref{eq_phi}, i.e., $V_{\textnormal{ext}}\circ\Phi_{g} = V_{\textnormal{ext}}$, for any $g \in G$. It is also easy to verify that the potential function is invariant under the left action of the isotropy group
\begin{align}
G_{\alpha_{0}}\colonequals\{g\in G\mid\Ad_{g}^{*}\alpha_{0} = \alpha_{0}\}\cong\mathrm{SO}(2),\nonumber
\end{align}
i.e., rotations about the vertical axis $\mathbf{e}_{3}$. We can now see that the assumptions \eqref{a1}--\eqref{a6} are satisfied. By Corollary \ref{cor5}, the Euler--Poincar{\'e} equations \eqref{eq_ep_thm4}--\eqref{eq_diff_ep_thm4} (under the identifications $\mathfrak{g}\cong\mathbb{R}^{3}$ and $\mathfrak{g}^{*}\cong\mathbb{R}^{3}$) give the following equations
\begin{align}
\mathbf{I}\dot{u} &= \ad_{u}^{*}\mathbf{I}u-\ad_{\frac{\partial V_{\textnormal{ext}}}{\partial\alpha}}^{*}\alpha,\label{eq_ep1_ht}\\
\dot{\alpha} &= \ad_{u}^{*}\alpha, \hspace{5pt} \alpha(0) = \Ad_{g_{0}}^{*}\alpha_{0},\label{eq_diff_ht}
\end{align}
where
\begin{align}
\ad_{u}^{*}\mathbf{I}u = \mathbf{I}u\times u, \hspace{5pt} \ad_{\frac{\partial V_{\textnormal{ext}}}{\partial\alpha}}^{*}\alpha = -m\mathbf{g}l\hspace{2pt}\alpha\times\boldsymbol{\chi}, \hspace{5pt} \ad_{u}^{*}\alpha = \alpha\times u, \hspace{5pt} \Ad_{g_{0}}^{*}\alpha_{0} = g_{0}^{-1}\alpha_{0}\nonumber
\end{align}
and $\mathbf{I}\in\mathbb{R}^{3\times 3}$ is the inertia matrix. For more details, see \cite{Holm2009}, \cite{Marsden1999}. So, \eqref{eq_ep1_ht}--\eqref{eq_diff_ht} give the following equations
\begin{align}
\mathbf{I}\dot{u} &= \mathbf{I}u\times u-m\mathbf{g}l\hspace{2pt}\boldsymbol{\chi}\times\alpha,\nonumber\\
\dot{\alpha} &= \alpha\times u, \hspace{5pt} \alpha(0) = g_{0}^{-1}\alpha_{0}.\nonumber
\end{align}
We do not give all the details and leave it up to reader to verify that the Lie--Poisson equations \eqref{eq_lp}--\eqref{eq_diff_lp} are given by 
\begin{align}
\dot{\mu} &= \mu\times\mathbf{I}^{-1}\mu-m\mathbf{g}l\hspace{2pt}\boldsymbol{\chi}\times\alpha,\nonumber\\
\dot{\alpha} &= \alpha\times\mathbf{I}^{-1}\mu, \hspace{5pt} \alpha(0) = g_{0}^{-1}\alpha_{0},\nonumber
\end{align}
which are the well known heavy top equations (see, e.g., \cite[Section 7.4]{Holm2009}, \cite[Section 1.4]{Marsden1999}).

\subsection{Motion Planning of a Unicycle with Obstacles}\label{subsec4.2}
We study the OCP for the motion planning of a unicycle with obstacles. To avoid the obstacle, we use the barrier function approach (see, e.g, \cite{Hauser2006}, \cite{Saccon2012}), which plays the role of the potential function in the cost function of the OCP.

Our model for a unicycle is a homogeneous disk rolling on a horizontal plane maintaining its vertical position (see, e.g., \cite[Section 1.4]{Bloch2015}). The configuration of the unicycle at any given time is completely determined by the element $g\in\mathrm{SE}(2)\cong\mathbb{R}^{2}\times\mathrm{S}^{1}\cong\mathbb{R}^{2}\times\mathrm{SO}(2)$ (as a set) given by
\begin{align}
g = \begin{bmatrix}
    \cos\theta & -\sin\theta & \phantom{-}x\\
    \sin\theta & \phantom{-}\cos\theta & \phantom{-}y\\
    0 & \phantom{-}0 & \phantom{-}1
  \end{bmatrix},\nonumber
\end{align}
where $[x \hspace{5pt} y]^{T}\in\mathbb{R}^{2}$ represents the point of contact of the wheel with the ground and $\theta\in\mathrm{S}^{1}$ represents the angular orientation of the overall system (see Figure \ref{Fig1}). For more details, see \cite{Bloch2015}, \cite{Krishnaprasad1993}. The controlled equations for the unicycle are given by
\begin{align}
\dot{x} &= u^{2}\cos\theta,\label{eq_uc1}\\
\dot{y} &= u^{2}\sin\theta,\label{eq_uc2}\\
\dot{\theta}&= u^{1}.\label{eq_uc3}
\end{align}
We can view \eqref{eq_uc1}--\eqref{eq_uc3} as a left-invariant control system on $\mathrm{SE}(2)$ (see \cite{Krishnaprasad1993}). Note that \eqref{eq_uc1}--\eqref{eq_uc2} are equivalent to the nonholonomic constraint $\dot{x}\sin\theta-\dot{y}\cos\theta = 0$.

\begin{figure}[h]
\centering
\usetikzlibrary{arrows}
\begin{tikzpicture}[>=latex']
\begin{scope}[scale=1]
\draw[->,thick] (0,0) to (4,0);
\draw[->,thick] (0,0) to (2,2);
\draw[dashed] (1,1) to (3.8,1);
\draw[dashed] (2,0) to (3,1);
\draw[fill=black] (3,2.5) -- (2.85,2.5) arc (90:270:0.4 and 0.75) -- (3,1) arc (270:90:0.4 and 0.75);
\draw[dashed] (3,1) -- ++(30:0.8);
\draw[black,fill=white] (3,1.75) ellipse (0.4 and 0.75);
\draw [->,thick] (3.565,1) arc (0:17.15:1);
\draw (3.565,1.2) node[right] {\large $\theta$};
\draw (1.95,-0.05) node[below] {\large $x$} to (2.05,0.05);
\draw (0.95,1) node[left] {\large $y$} to (1.05,1);
\end{scope}
\end{tikzpicture}
\caption{The Unicycle.}\label{Fig1}
\end{figure}
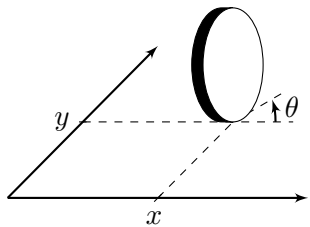

Let the obstacle be circular in shape and be located in the $x$-$y$ plane. Without loss of generality, assume that the obstacle has unit radius, with its center located at the point $(0,0)$. Let the potential function $V:\mathbb{R}^{2}\setminus S_{1}\to\mathbb{R}$ be given by 
\begin{align}
V(x,y) = \frac{\kappa}{2(x^{2}+y^{2}-1)},\nonumber
\end{align}
where $S_{1}\colonequals\{(x,y)\in\mathbb{R}^{2}\mid x^{2}+y^{2} = 1\}$ and $\kappa\in\mathbb{R}_{\geq 0}$. Let us now equip $\mathfrak{se}(2)$ with the inner product $\langle\xi,\xi\rangle\colonequals\tr(\xi^{T}\xi)$, for any $\xi\in\mathfrak{se}(2)$ and hence, the norm $\|\xi\|\colonequals\langle\xi,\xi\rangle^{1/2} = \sqrt{\tr(\xi^{T}\xi)}$, for any $\xi\in\mathfrak{se}(2)$. Equivalently, the potential function $V: \mathrm{SE}(2)\setminus S_{2}\to\mathbb{R}$ is given by
\begin{align}
V(g) = \frac{\kappa}{2(\|\Ad_{g^{-1}}e_{1}\|^{2}-1)},\nonumber
\end{align}
where $S_{2}\colonequals\{g\in\mathrm{SE}(2)\mid\|\Ad_{g^{-1}}e_{1}\| = 1\}$ and $e_{1}$ is given below. We now consider $\textbf{(P1)}$, with $G = \mathrm{SE}(2)$, $e_{0} = 0_{3\times 3}$ and with the cost function given by
\begin{align}
C(g,u) &= \frac{1}{2}\langle u,u\rangle,\nonumber\\
V(g) &= \frac{\kappa}{2(\|\Ad_{g^{-1}}e_{1}\|^{2}-1)},\nonumber
\end{align}
where $u(\cdot) = \sum_{i=1}^{2}u^{i}(\cdot)e_{i}$, with the elements of the basis of $\mathfrak{g} = \mathfrak{se}(2)$ (which is not semisimple) given by
\begin{align}
e_{1} = \begin{bmatrix}
        0 & -1 & \phantom{-}0\\
        1 & \phantom{-}0 & \phantom{-}0\\
        0 & \phantom{-}0 & \phantom{-}0
        \end{bmatrix}, \hspace{5pt}
e_{2} = \begin{bmatrix}
        0 & \phantom{-}0 & \phantom{-}1\\
        0 & \phantom{-}0 & \phantom{-}0\\
        0 & \phantom{-}0 & \phantom{-}0
        \end{bmatrix}, \hspace{5pt}
e_{3} = \begin{bmatrix}
        0 & \phantom{-}0 & \phantom{-}0\\
        0 & \phantom{-}0 & \phantom{-}1\\
        0 & \phantom{-}0 & \phantom{-}0
        \end{bmatrix},\nonumber
\end{align}
which satisfy
\begin{align}
[e_{1},e_{2}] = e_{3}, \hspace{5pt} [e_{2},e_{3}] = 0_{3\times 3}, \hspace{5pt} [e_{3},e_{1}] = e_{2}.\nonumber
\end{align}
Under the dual pairing, where $\langle\alpha,\xi\rangle\colonequals\tr(\alpha\xi)$, for any $\xi\in\mathfrak{g}$ and $\alpha\in\mathfrak{g}^{*} = \mathfrak{se}(2)^{*}$, the elements of the basis of $\mathfrak{g}^{*}$ are given by
\begin{align}
e^{1} = \begin{bmatrix}
        \phantom{-}0 & \phantom{-}\dfrac{1}{2} & \phantom{-}0\\
        -\dfrac{1}{2} & \phantom{-}0 & \phantom{-}0\\[8pt]
        \phantom{-}0 & \phantom{-}0 & \phantom{-}0
        \end{bmatrix}, \hspace{5pt}
e^{2} = \begin{bmatrix}
        0 & \phantom{-}0 & \phantom{-}0\\
        0 & \phantom{-}0 & \phantom{-}0\\
        1 & \phantom{-}0 & \phantom{-}0
        \end{bmatrix}, \hspace{5pt}
e^{3} = \begin{bmatrix}
        0 & \phantom{-}0 & \phantom{-}0\\
        0 & \phantom{-}0 & \phantom{-}0\\
        0 & \phantom{-}1 & \phantom{-}0
        \end{bmatrix}.\nonumber
\end{align}
It is easy to verify that the potential function is $\mathrm{SO}(2)$-invariant but not $G$-invariant and so, the potential function breaks the symmetry partially. We will now show how the theory developed in the paper can be applied to the OCP under consideration.

Let $X = \mathfrak{g}^{*}$ and $\rho$ be the coadjoint representation of $G$ on $X$, i.e., $\rho_{g} = \Ad_{g^{-1}}^{*}$, for any $g\in G$. So, $\rho^{*}$ is the adjoint representation of $G$ on $X^{*} = \mathfrak{g}^{**}\cong\mathfrak{g}$, i.e., $\rho^{*}_{g} = \Ad_{g^{-1}}$, for any $g\in G$. Let the extended potential function $V_{\textnormal{ext}}: G\times X^{*}\to\mathbb{R}$ be given as follows
\begin{align}
V_{\textnormal{ext}}(g,\alpha) = \frac{\kappa}{2(\|\Ad_{g^{-1}}\alpha\|^{2}-1)}.\nonumber
\end{align}
If we set $\alpha_{0} = e_{1}$ in the extended potential function, we recover our original potential function. It is easy to verify that the extended potential function is $G$-invariant under \eqref{eq_phi}, i.e., $V_{\textnormal{ext}}\circ\Phi_{g} = V_{\textnormal{ext}}$, for any $g \in G$. It is also easy to verify that the potential function is invariant under the left action of the isotropy group
\begin{align}
G_{\alpha_{0}}\colonequals\{g\in G\mid\Ad_{g^{-1}}\alpha_{0} = \alpha_{0}\} = \left\{\begin{bmatrix}\cos\theta & -\sin\theta &\phantom{-}0\\\sin\theta & \phantom{-}\cos\theta & \phantom{-}0\\0 & \phantom{-}0 & \phantom{-}1\end{bmatrix}\hspace{4pt}\middle|\hspace{4pt}\theta\in\mathrm{S}^{1}\right\}\cong\mathrm{SO}(2).\nonumber
\end{align}
We can now see that the assumptions \eqref{a1}--\eqref{a6} are satisfied and we also have $\mathfrak{g} = \mathfrak{k}\oplus\mathfrak{p}$ such that 
\begin{align}
[\mathfrak{k},\mathfrak{k}]\subseteq\mathfrak{p}, \hspace{5pt} [\mathfrak{p},\mathfrak{k}]\subseteq\mathfrak{k}, \hspace{5pt} [\mathfrak{p},\mathfrak{p}]\subseteq\mathfrak{p}.\nonumber
\end{align}
By Corollary \ref{cor6} and Proposition \ref{prop7}, the Euler--Poincar{\'e} equations \eqref{eq_ep_thm4}--\eqref{eq_diff_ep_thm4} give the following equations
\begin{align}
\dot{u} &= \ad_{u}^{*}\lambda+\ad_{\alpha}^{*}\frac{\partial V_{\textnormal{ext}}}{\partial\alpha}\bigg|_{\mathfrak{k}},\label{eq_ep1_uc}\\
\dot{\lambda} &= \ad_{u}^{*}u+\ad_{\alpha}^{*}\frac{\partial V_{\textnormal{ext}}}{\partial\alpha}\bigg|_{\mathfrak{p}},\label{eq_ep2_uc}\\
\dot{\alpha} &= -\ad_{u}\alpha, \hspace{5pt} \alpha(0) = \Ad_{g_{0}^{-1}}\alpha_{0},\label{eq_diff_uc}
\end{align}
where 
\begin{align}
\ad_{u}^{*}\lambda &= \begin{bmatrix} 
                      0 & -\dfrac{u^{2}\lambda_{3}}{2} & \phantom{-}0\\
										  \dfrac{u^{2}\lambda_{3}}{2} & \phantom{-}0 & \phantom{-}0\\[8pt] 
										  u^{1}\lambda_{3} & \phantom{-}0 & \phantom{-}0
										  \end{bmatrix}, \hspace{5pt}
\ad_{\alpha}^{*}\frac{\partial V_{\textnormal{ext}}}{\partial\alpha} = 
\frac{\kappa\alpha^{1}}{(\|\alpha\|^{2}-1)^{2}}\begin{bmatrix}
                                               \phantom{-}0 & \phantom{-}0 & \phantom{-}0\\
                                               \phantom{-}0 & \phantom{-}0 & \phantom{-}0\\
                                               -\alpha^{3} & \phantom{-}\alpha^{2} & \phantom{-}0
                                               \end{bmatrix},\nonumber\\										
&\ad_{u}^{*}u = \begin{bmatrix} 
                0 & \phantom{-}0 & \phantom{-}0\\
							  0 & \phantom{-}0 & \phantom{-}0\\
							  0 & -u^{1}u^{2} & \phantom{-}0
							  \end{bmatrix}, \hspace{5pt}
\ad_{u}\alpha = \begin{bmatrix}
                0 & \phantom{-}0 & -u^{1}\alpha^{3}\\
                0 & \phantom{-}0 & \phantom{-}u^{1}\alpha^{2}-u^{2}\alpha^{1}\\
                0 & \phantom{-}0 & \phantom{-}0
                \end{bmatrix}\nonumber
\end{align}								
and
\begin{align}								
\Ad_{g_{0}^{-1}}\alpha_{0} = \begin{bmatrix}
                             0 & -1 & \phantom{-}x_{0}\sin\theta_{0}-y_{0}\cos\theta_{0}\\
                             1 & \phantom{-}0 & \phantom{-}x_{0}\cos\theta_{0}+y_{0}\sin\theta_{0}\\
                             0 & \phantom{-}0 & \phantom{-}0
                             \end{bmatrix}.\nonumber
\end{align}
For more details, see \cite[Section 14.6]{Marsden1999}. So, \eqref{eq_ep1_uc}--\eqref{eq_diff_uc} give the following equations
\begin{align}
\dot{u}^{1} &= -\frac{u^{2}\lambda_{3}}{2},\nonumber\\
\dot{u}^{2} &= u^{1}\lambda_{3}-\frac{\kappa\alpha^{1}\alpha^{3}}{(\|\alpha\|^{2}-1)^{2}},\nonumber\\
\dot{\lambda}_{3} &= -u^{1}u^{2}+\frac{\kappa\alpha^{1}\alpha^{2}}{(\|\alpha\|^{2}-1)^{2}},\nonumber\\
\dot{\alpha}^{1} &= 0, \hspace{5pt} \alpha^{1}(0) = 1,\nonumber\\
\dot{\alpha}^{2} &= u^{1}\alpha^{3}, \hspace{5pt} \alpha^{2}(0) = x_{0}\sin\theta_{0}-y_{0}\cos\theta_{0},\nonumber\\
\dot{\alpha}^{3} &= -u^{1}\alpha^{2}+u^{2}\alpha^{1}, \hspace{5pt} \alpha^{3}(0) = x_{0}\cos\theta_{0}+y_{0}\sin\theta_{0}.\nonumber
\end{align}
We do not give all the details and leave it up to reader to verify that the Lie--Poisson equations \eqref{eq_lp}--\eqref{eq_diff_lp} are given by 
\begin{align}
\dot{\mu}_{1} &= -\mu_{2}\mu_{3},\nonumber\\
\dot{\mu}_{2} &= \frac{\mu_{1}\mu_{3}}{2}-\frac{\kappa\alpha^{1}\alpha^{3}}{(\|\alpha\|^{2}-1)^{2}},\nonumber\\
\dot{\mu}_{3} &= -\frac{\mu_{1}\mu_{2}}{2}+\frac{\kappa\alpha^{1}\alpha^{2}}{(\|\alpha\|^{2}-1)^{2}},\nonumber\\
\dot{\alpha}^{1} &= 0, \hspace{5pt} \alpha^{1}(0) = 1,\nonumber\\
\dot{\alpha}^{2} &= \frac{\mu_{1}\alpha^{3}}{2}, \hspace{5pt} \alpha^{2}(0) = x_{0}\sin\theta_{0}-y_{0}\cos\theta_{0},\nonumber\\
\dot{\alpha}^{3} &= -\frac{\mu_{1}\alpha^{2}}{2}+\mu_{2}\alpha^{1}, \hspace{5pt} \alpha^{3}(0) = x_{0}\cos\theta_{0}+y_{0}\sin\theta_{0}.\nonumber
\end{align}

\begin{remark}
Let us consider the case, when the OCP under consideration has no potential function. In this case, the Euler--Poincar{\'e} equations \eqref{eq_ep_thm4} give the following equations
\begin{align}
\dot{u}^{1} &= -\frac{u^{2}\lambda_{3}}{2},\label{eq_ep_np_1}\\
\dot{u}^{2} &= u^{1}\lambda_{3},\\
\dot{\lambda}_{3} &= -u^{1}u^{2}.\label{eq_ep_np_3}
\end{align}
We will now show that the Euler--Poincar{\'e} equations \eqref{eq_ep_np_1}--\eqref{eq_ep_np_3} are exactly the Euler--Lagrange equations, when we view the OCP under consideration as a constrained variational problem. The OCP under consideration is equivalent to the following constrained variational problem
\begin{align}
&\min_{([x(\cdot) \hspace{5pt} y(\cdot)]^{T},\theta(\cdot))}J\colonequals\min_{([x(\cdot) \hspace{5pt} y(\cdot)]^{T},\theta(\cdot))}\frac{1}{2}\int_{0}^{T}[\dot{x}^{2}(t)+\dot{y}^{2}(t)+2\dot{\theta}^{2}(t)]dt\nonumber
\intertext{subject to\hfill\textbf{(P3)}}
&\dot{x}(t)\sin\theta(t)-\dot{y}(t)\cos\theta(t) = 0, \hspace{5pt} \text{with given boundary conditions} \hspace{5pt} ([x(0) \hspace{5pt} y(0)]^{T},\theta(0))\nonumber\\
&\text{and} \hspace{5pt} ([x(T) \hspace{5pt} y(T)]^{T},\theta(T)).\nonumber
\end{align}
The Lagrangian for $\textbf{(P3)}$ is given by
\begin{align}
L_{a}(\theta,\dot{x},\dot{y},\dot{\theta},\lambda) = \frac{1}{2}(\dot{x}^{2}+\dot{y}^{2}+2\dot{\theta}^{2})+\lambda(\dot{y}\cos\theta-\dot{x}\sin\theta),\nonumber
\end{align}
where $\lambda$ is the Lagrange multiplier. A solution for $\textbf{(P3)}$ must satisfy the following Euler--Lagrange equations
\begin{align}
\ddot{x} &= \dot{\lambda}\sin\theta-\lambda\dot{\theta}\cos\theta,\label{eq_el_1}\\
\ddot{y} &= \dot{\lambda}\cos\theta-\lambda\dot{\theta}\sin\theta,\\
\ddot{\theta} &= \frac{\lambda}{2}(\dot{x}\cos\theta+\dot{y}\sin\theta).\label{eq_el_3}
\end{align}
Using the facts that $\dot{x}\sin\theta-\dot{y}\cos\theta = 0$, $u^{1} = \dot{\theta}$, $u^{2} = \dot{x}\cos\theta+\dot{y}\sin\theta$ and after a few simple calculations, the Euler--Lagrange equations \eqref{eq_el_1}--\eqref{eq_el_3} give the following equations
\begin{align}
\dot{u}^{1} &= -\frac{u^{2}\lambda}{2},\nonumber\\
\dot{u}^{2} &= u^{1}\lambda,\nonumber\\
\dot{\lambda} &= -u^{1}u^{2},\nonumber
\end{align}
which are exactly the Euler--Poincar{\'e} equations \eqref{eq_ep_np_1}--\eqref{eq_ep_np_3}.
\end{remark}

\begin{remark}
For the above example, in the discrete-time setting, one would choose
\begin{align}
C_{d}(g_{k},u_{k}) &= \frac{h}{2}\langle u_{k},u_{k}\rangle,\nonumber\\
V_{d}(g_{k}) &= \frac{h\kappa}{2(\|\Ad_{g_{k}^{-1}}e_{1}\|^{2}-1)},\nonumber
\end{align}
where 
\begin{align}
g_{k} = \begin{bmatrix}
        \cos\theta_{k} & -\sin\theta_{k} & \phantom{-}x_{k}\\
        \sin\theta_{k} & \phantom{-}\cos\theta_{k} & \phantom{-}y_{k}\\
        0 & \phantom{-}0 & \phantom{-}1
        \end{bmatrix}\in G\nonumber
\end{align}
and $u_{k} = \sum_{i=1}^{2}u_{k}^{i}e_{i}\in\mathfrak{g}$. Also, in the discrete-time setting, the extended potential function $V_{d,\textnormal{ext}}: G\times X^{*}\to\mathbb{R}$ can be constructed in exactly the same way as in the above example and is given by
\begin{align}
V_{d,\textnormal{ext}}(g_{k},\alpha) = \frac{h\kappa}{2(\|\Ad_{g_{k}^{-1}}\alpha\|^{2}-1)}.\nonumber
\end{align}
We do not give all the details again and leave it up to reader to verify that the assumptions \eqref{a1}--\eqref{a6} are satisfied. By Corollary \ref{cor13}, the discrete-time Lie--Poisson equations \eqref{eq_dlp_thm11}-\eqref{eq_diff_dlp_thm11} give the following equations
\begin{align}
(d\tau_{hu_{k}}^{-1})^{*}\mu_{k} &= (d\tau_{-hu_{k-1}}^{-1})^{*}\mu_{k-1}+\ad_{\bar{\alpha}_{k}}^{*}\frac{\partial V_{d,\textnormal{ext}}}{\partial\bar{\alpha}_{k}},\nonumber\\
\bar{\alpha}_{k+1} &= \Ad_{\tau(hu_{k})^{-1}}\bar{\alpha}_{k}, \hspace{5pt} \bar{\alpha}_{0} = \Ad_{g_{0}^{-1}}\alpha_{0}\nonumber,
\end{align}
where
\begin{align}
\ad_{\bar{\alpha}_{k}}^{*}\frac{\partial V_{d,\textnormal{ext}}}{\partial\bar{\alpha}_{k}} 
= \frac{h\kappa\bar{\alpha}_{k}^{1}}{(\|\bar{\alpha}_k\|^{2}-1)^{2}}\begin{bmatrix}
                                                                    \phantom{-}0 & \phantom{-}0  & \phantom{-}0\\
                                                                    \phantom{-}0  & \phantom{-}0 & \phantom{-}0\\
                                                                    -\bar{\alpha}_{k}^{3} & \phantom{-}\bar{\alpha}_{k}^{2} & \phantom{-}0
                                                                    \end{bmatrix}\nonumber
\end{align}
and $\mu_{k} = u_{k}\in\mathfrak{g}^{*}$. For numerical purposes, one first chooses an appropriate retraction map, such as the Cayley map or the exponential map, etc. (see, e.g., \cite{BouRabee2007}, \cite{BouRabee2009}, \cite{Kobilarov2009}, \cite{Kobilarov2007}, \cite{Kobilarov2011}) and then computes the quantities $(d\tau_{hu_{k}}^{-1})^{*}\mu_{k}$ and $(d\tau_{-hu_{k-1}}^{-1})^{*}\mu_{k-1}$. As an example, if we choose the Cayley map $\text{cay}: \mathfrak{se}(2)\to\mathrm{SE}(2)$ as the retraction map, then we have
\begin{align}
[d\text{cay}_{hu_{k}}^{-1}]^{*}\mu_{k} = \begin{bmatrix}
                                         -\dfrac{h^{2}(u_{k}^{1})^{2}u_{k}^{2}}{4}+\dfrac{h(u_{k}^{2})^{2}}{2} & \phantom{-}u_{k}^{1}+
																				 \dfrac{h^{2}(u_{k}^{1})^3}{4} & \phantom{-}0\\
                                         -u_{k}^{1}-\dfrac{hu_{k}^{1}u_{k}^{2}}{2} & \phantom{-}0 & \phantom{-}0\\
                                         -\dfrac{h(u_{k}^{1})^{2}}{2}+u_{k}^{2} & \phantom{-}0 & \phantom{-}0
                                         \end{bmatrix}\nonumber
\end{align}																	
and
\begin{align}
[d\text{cay}_{-hu_{k-1}}^{-1}]^{*}\mu_{k-1} = \begin{bmatrix}
                                              -\dfrac{h^{2}(u_{k-1}^{1})^{2}u_{k-1}^{2}}{4}-\dfrac{h(u_{k-1}^{2})^{2}}{2} & \phantom{-}
																							u_{k-1}^{1}+\dfrac{h^{2}(u_{k-1}^{1})^3}{4} & \phantom{-}0\\
                                              -u_{k-1}^{1}+\dfrac{hu_{k-1}^{1}u_{k-1}^{2}}{2} & \phantom{-}0 & \phantom{-}0\\
                                              \phantom{-}\dfrac{h(u_{k-1}^{1})^{2}}{2}+u_{k-1}^{2} & \phantom{-}0 & \phantom{-}0
                                              \end{bmatrix}.\nonumber
\end{align}
Note that for $v = \sum_{i=1}^{3}v^{i}e_{i}\in\mathfrak{g}$, the matrix representation for $d\text{cay}_{v}^{-1}$ is given by
\begin{align}
[d\text{cay}_{v}^{-1}] = \begin{bmatrix}
                         1+\dfrac{(v^1)^2}{4} & \phantom{-}0  & \phantom{-}0\\
                         \dfrac{v^1v^2}{4}-\dfrac{v^3}{2} & \phantom{-}1 & \phantom{-}\dfrac{v^1}{2}\\
                         \dfrac{v^1v^3}{4}+\dfrac{v^2}{2} & -\dfrac{v^1}{2} & \phantom{-}1
                         \end{bmatrix}.\nonumber
\end{align}
For more details, see \cite{BouRabee2007}, \cite{BouRabee2009}, \cite{Kobilarov2009}, \cite{Kobilarov2007}, \cite{Kobilarov2011}.
\end{remark}

\section{Conclusions}\label{sec5}
We studied symmetry reduction of OCPs for left-invariant control systems on Lie groups, with partial symmetry breaking cost functions and we obtained the Euler--Poincar{\'e} equations. Furthermore, by using a Legendre transformation, we obtained the Lie--Poisson equations and in the discrete-time setting, we obtained the discrete-time Lie--Poisson equations. The theory was also illustrated with some practical examples.

\section*{Acknowledgments}
The research of Anthony M. Bloch was supported by NSF grants DMS-$1207693$, DMS-$1613819$, INSPIRE-$1343720$ and the Simons Foundation. The research of Leonardo J. Colombo was supported by MINECO (Spain) grant MTM $2013$-$42870$-P and NSF grant INSPIRE-$1343720$.

\end{document}